\documentclass[a4paper,12pt]{article}
\usepackage[utf8x]{inputenc}
\usepackage{graphicx}
\usepackage{amsmath}
\usepackage{amssymb}
\usepackage{amsthm}
\usepackage{enumerate}
\usepackage{txfonts}

\theoremstyle{definition}
\newtheorem{thm}{Theorem}
\newtheorem{dfn}[thm]{Definition}

\newtheorem{lem}[thm]{Lemma}
\newtheorem{pps}[thm]{Proposition}
\newtheorem{es}[thm]{Example}
\newtheorem{thmi}{Theorem}

\newtheorem{cori}[thmi]{Corollary}
\newtheorem{lemi}[thmi]{Lemma}
\newtheorem{ppsi}[thmi]{Proposition}

\newcommand{\C}{\mathbb C}
\newcommand{\Sp}{\mathfrak{S}}
\newcommand{\inv}[1]{{#1^{-1}}}
\newcommand{\cl}{\mbox{cl}}

\newcommand{\Ccg}[1]{{\overline{#1}}}
\newcommand{\lcm}{\mbox{lcm}}
\newcommand{\gen}[1]{\langle #1 \rangle}
\newcommand{\Zel}[1]{\mathfrak B_{p^{#1}}}
\newcommand{\Hom}{\mbox{Hom}}
\newcommand{\Hn}{\mbox{H}}
\newcommand{\res}{\mbox{res~}}
\renewcommand{\inf}{\mbox{inf~}}

\newcommand{\M}{\mbox{M}}
\newcommand{\ZU}{Z_{\mbox{\scriptsize u}}}
\newcommand{\BU}{B_{\mbox{\scriptsize u}}}
\newcommand{\GU}{\Gamma_{\mbox{\scriptsize u}}}
\newcommand{\idempotent}[1]{\varepsilon_#1}
\newcommand{\omegatimes}[2]{\check{#1}\propto{#2}}

\title{The unitary cover of a finite group and\\ the exponent of the Schur multiplier.}
\author{
Nicola Sambonet\footnote{Department of Mathematics, Technion, Haifa, 32000 Israel, sambonet@tx.technion.ac.il.}\vspace{4mm}\\
In memory of David Chillag.}
\date{\vspace{-5ex}}
\begin{document}

\maketitle

\begin{abstract}
For a finite group we introduce a particular central extension, the unitary cover, having minimal exponent among those satisfying the projective lifting property. We obtain new bounds for the exponent of the Schur multiplier relating to subnormal series, and we discover new families for which the bound is the exponent of the group.
Finally, we show that unitary covers are controlled by the Zel'manov solution of the restricted Burnside problem for 2-generator groups.

\end{abstract}

\section{Introduction}
The \emph{Schur multiplier} of a finite group $G$ is the second cohomology group with complex coefficients, denoted by $\M(G)=\Hn^2(G,\C^\times)$.
It was introduced in the beginning of the twentieth century by I.~Schur, aimed at the study of projective representations.
To determine $\M(G)$ explicitly is often a difficult task.
Therefore, it is of interest to provide bounds for numerical qualities of $\M(G)$ as the order, the rank, and - our subject - the exponent.

In 1904 Schur already showed that $[\exp\M(G)]^2$ divides the order of the group, and this bound is tight as $\M(C_n\times C_n)=C_n$.
Note that $C_n\times C_n$ is an example of group satisfying
\begin{equation}\label{eq exp MG divides exp G}
 \exp\M(G)\ \ |\ \ \exp G\ ,
\end{equation}
property which has been proven for many classes of groups.

\subsection{Groups such that exp M(G) divides exp G}\label{section exp G divides exp MG}
Firstly, \eqref{eq exp MG divides exp G} holds for every abelian group $G$.
Indeed, consider the cyclic decomposition ordered by recursive division:
\begin{equation}\label{eq cyclic decomposition}
 G=\overset{n}{\underset{i=1}{\oplus}} C_{d_i}\ \ ,\ \ d_i\ |\ d_{i+1}\ .
\end{equation}
 By Schur it is known (cf.~\cite[p.~317]{karpilovsky}) that 
\begin{equation}\label{eq the multiplier of an abelian group}
 \M(G)=\overset{n}{\underset{i=1}{\oplus}}{C_{d_i}}^{n-i}\ .
\end{equation}
Consequently, $\exp\M(G)=d_{n-1}$ which in turn divides $\exp G=d_n$.
A second important example of groups enjoying \eqref{eq exp MG divides exp G} are the finite simple groups, whose multipliers are known and listed in the Atlas \cite{FSGC}.

A standard argument (cf.~\cite[Th.10.3]{brown}) proposes to focus on $p$-groups.
Indeed, the $p$-component of $\M(G)$ is embedded in the multiplier of a $p$-Sylow via the restriction map.
Therefore, if $\Pi(G)$ denotes the set of prime divisors of $G$, and $S_p$ denotes a $p$-Sylow of $G$ for $p\in\Pi(G)$, then
\begin{equation}\label{eq p-Sylows}
\exp\M(G)\ |\ \prod_{p\in\Pi(G)}\exp\M(S_p)\ .
\end{equation}
Clearly, since $$\exp G=\prod_{p\in\Pi(G)}\exp S_p\ ,$$ if \eqref{eq exp MG divides exp G} holds for every $p$-Sylow of $G$, then it does also for $G$.

Then, a fundamental feature of $p$-groups is the nilpotency class.
Recently, P.~Moravec completed a result of M.~R.~Jones \cite[Rem.~2.8]{jones74} proving \eqref{eq exp MG divides exp G} for groups of class at most 3, and extended this result to groups of class 4 in the odd-order case \cite[Th.~12, Th.~13]{moravecSg}.
Moravec discovered many other families enjoying \eqref{eq exp MG divides exp G}: metabelian groups of prime exponent \cite[Pr.~2.12]{moravec}, 3-Engel groups, 4-Engel groups in case the order is coprime with 2 and 5 \cite[Cor.~5.5, Cor.~4.2]{moravecEngel}, $p$-groups of class lower than $p-1$ \cite[Pr.~11]{moravecSg}, and $p$-groups of maximal class \cite[Th.~1.4]{moravecCoclass}.
Without pretending to complete a list, we mention that \eqref{eq exp MG divides exp G} holds for extraspecial and abelian-by-cyclic groups, as follows from an argument of R.J.~Higgs \cite[Pr.~2.3, Pr.~2.4]{higgs}.

We have not cited 2-groups of class 4.
Actually, the general validity of \eqref{eq exp MG divides exp G} was disproved by such a group long time before all the reported examples.
A.~J.~Bayes, J.~Kautsky and J.~W.~Wamsley introduced a group of order $2^{68}$ and exponent 4 whose multiplier has exponent 8 \cite{bayes}.
Lately, Moravec described another counterexample of order $2^{11}$ and class 6 \cite[Ex.~2.9]{moravec}.
Nevertheless, these essentially are the only counterexamples we know: both were obtained by computer technique and satisfy $\exp\M(G)=8$ but $\exp G=4$.

The scenario in case $p>2$ is indeed not clear yet.
For instance, groups of exponent 3 satisfy \eqref{eq exp MG divides exp G} as they have nilpotency class at most 3 (cf.~\cite[p.~6]{moravecEngel}). This family, as well as the metabelian groups of prime exponent which we already encountered, strengthens the idea that \eqref{eq exp MG divides exp G} should at least holds for groups of prime exponent.

A.~Lubotzky and A.~Mann proved \eqref{eq exp MG divides exp G} for powerful $p$-groups \cite[Th.~2.4]{lubotzkymann}.
By definition, a $p$-group $G$ for $p>2$ is \emph{powerful} if the derived subgroup $G'$ is contained in the agemo subgroup $$\mho(G)=\gen{g^p\ |\ g\in G}$$ generated by the $p$-powers.
Abelian $p$-groups are powerful, and if a $p$-group is powerful, then its quotients of exponent $p$ are necessarily abelian.
For this reason, powerful $p$-group and groups of exponent $p$ can be considered as two extremes dealing with $p$-groups for $p>2$.

\subsection{Bounds for exp M(G)}
Beside the result on powerful $p$-groups, Lubotzky and Mann provided a bound for $\exp\M(G)$ involving the exponent and the rank of $G$ \cite[Pr.~2.6,Pr.~4.2.6]{lubotzkymann}, this has been recently refined by J.~Gonz\'alez S\'anchez and A.~P.~Nicolas \cite[Th.~2]{sanchez}.

Meanwhile, Moravec proved the existence of a bound only in terms of the exponent of the group \cite[Pr.~2.4]{moravec}.
This bound relies on the Zel'manov solution of the restricted Burnside problem, with the idea that the problem of finding bounds for the exponent of the multiplier can be reduced in some extent to 2-generator groups, we will give an alternative proof of this fact.

Since the use of the Zel'manov solution gives a bound which is apparently far from being efficient, Moravec also stated a more practical bound:
\begin{equation}\label{eq moravec bound}
 \exp\M(G)\ \ |\ \ (\exp G)^{2(d-1)}
\end{equation}
where $d$ is the derived length assumed to be greater than 1 \cite[Th.~2.13]{moravec}.
The bound analogue to \eqref{eq moravec bound} with the nilpotency class $c$ in place of $d$ was previously discovered by Jones \cite[Cor.~2.7]{jones74}, then modified as
\begin{equation}\label{eq ellis bound}
 \exp\M(G)\ \ |\ \ (\exp G)^{\lceil c/2\rceil}
\end{equation}
by G.~Ellis \cite[Th.~B1]{ellis}.
As Moravec illustrated, \eqref{eq moravec bound} improves \eqref{eq ellis bound} for $c\geq 11$ via the formula
\begin{equation}\label{eq derived length and nilpotency class}
 d\leq\lfloor\log_2 c\rfloor+1
\end{equation}
relating nilpotency class and derived length (cf.~\cite[5.1.11]{robinson}).

\section{Results}
To give evidence for the content of this paper, we present some advancement concerning the problems exposed in the introduction.
We improve \eqref{eq moravec bound} and consequently \eqref{eq ellis bound}, also including the case of abelian groups for which \eqref{eq exp MG divides exp G} holds.
\begin{thmi}\label{thm derived length}
Let $G$ be a $p$-group of derived length $d$. Then
 $$\exp\M(G)\ \ |\ \ \left\{\begin{array}{rl}
 2^{d-1}\cdot(\exp G)^d & p=2\\
 (\exp G)^d & p>2
 \end{array}\right.$$
\end{thmi}

\emph{Comparison with \eqref{eq moravec bound}}:
in case $p>2$, the bounds coincide for $d=2$ and the improvement occurs for $d>2$;
in case $p=2$, it is non-efficient for $d=2$, the bounds coincide for $d=3$ and $\exp G=4$, and the improvement occurs in all the other cases. 

\emph{Comparison with \eqref{eq ellis bound} via \eqref{eq derived length and nilpotency class}}:
in case $p>2$, the bounds coincide for $c=4,5,6$ and $c=8$, and the improvement occurs for $c=7$ or $c\geq 9$;
in case $p=2$ and $\exp G=4$, the bounds coincide for $c=7$, and the improvement occurs for $c\geq 11$;
in case $p=2$ and $\exp G>4$, the bounds coincide for $c=7$, and the improvement occurs for $c\geq 9$.

The difference between the odd and the even case in Theorem \ref{thm derived length} can be explained with the concept of unitary cover (\S\ref{section technique}), based on the theory of central extensions and projective representations (cf.~\cite[\S 9]{isaacs}, and \S\ref{section background} hereby).
\begin{thmi}\label{thm main}
There exists a canonical element $\GU(G)$, the \emph{unitary cover} of $G$, which has minimal exponent in the set of central extensions of $G$ satisfying the projective lifting property.
The map $\GU$, associating to a group its unitary cover, satisfies for any normal subgroup $N$ of $G$ the following properties:
\begin{enumerate}[i)]
 \item $\exp\M(G)\ \ |\ \ \exp\GU(N)\cdot\exp\M(G/N)$
 \item $\exp\GU(G)\ \ |\ \ \exp\GU(N)\cdot\exp\GU(G/N)$
 \item $\GU(G/N)$ is a homomorphic image of $\GU(G)$ .
\end{enumerate}
Moreover, if $G=N\rtimes H$, then
\begin{enumerate}[i)]
 \item[iv)] $\exp\M(G)\ \ |\ \ \lcm\{\exp\GU(N),\ \exp\M(H)\}$
 \item[v)] $\exp\GU(G)\ \ |\ \ \lcm\{\exp\GU(N)\cdot\exp H,\ \exp\GU(H)\}$ .
\end{enumerate}
\end{thmi}
By minimality, one can eventually replace the unitary cover with any central extension satisfying the projective lifting property, for instance with any Schur cover.
The word ``canonical'' refers to the fact that $\GU(G)$ is uniquely defined, whereas two Schur covers need not to be isomorphic.

We determine the exponent of the unitary cover for abelian $p$-groups, and in case $p>2$ we extend this result for powerful $p$-groups (introduced in \S \ref{section exp G divides exp MG}).
Readily, we describe other families of groups for which \eqref{eq exp MG divides exp G} holds by Theorem \ref{thm main}.
\begin{lemi}\label{lem abelian 2-groups and powerful p-groups}
The following holds.
\begin{enumerate}[i)]
 \item If $G$ is a powerful $p$-group for $p>2$, then $\exp\GU(G)=\exp G$.
 \item If $G$ is an abelian $2$-group of exponent $n$.
Then $$\exp\GU(G)=2^\sigma\cdot\exp G$$ for \hspace{5mm} $\sigma=\left\{\begin{array}{ll}1&\mbox{if $G$ has a subgroup isomorphic with $C_n\times C_n$}\\0&\mbox{otherwise.}\end{array}\right.$
\end{enumerate}
\end{lemi}
\begin{cori}\label{cori new family}
Let $G$ be a $p$-group, and $N$ a normal subgroup of $G$. Assume one of the following:
$$\left\{\begin{array}{ll}
          p=2&\mbox{$N$ is abelian with no subgroups isomorphic with $C_{2^k}\times C_{2^k}$,}\\&\mbox{where $2^k$ is the exponent of $G$.}\\
	  p>2&\mbox{$N$ is a powerful $p$-group}.
         \end{array}\right.$$
And assume one of the following:
$$\left\{\begin{array}{l}\M(G/N)=1\ .\\G=N\rtimes H \mbox{ where $H$ satisfies \eqref{eq exp MG divides exp G}.}\hspace{57mm}\end{array}\right.$$
Then $G$ satisfies property \eqref{eq exp MG divides exp G}.
\end{cori}
At least for groups of odd order, the previous result generalizes the case of abelian-by-cyclic groups to powerful-by-trivial multiplier groups, and it reveals a closure property under semidirect products with powerful kernels.

Our next result concerns regular $p$-groups, which constitute one of the most important family of $p$-groups and were introduced by P.~Hall in 1934 \cite[\S 4]{hall}.
A $p$-group $G$ is \emph{regular} if for every $x,y\in G$ there exist $c\in\gen{x,y}'$ such that $(xy)^p=x^py^pc^p$.
Abelian $p$-groups are regular, regular 2-groups are abelian, and regular $p$-groups share important properties with abelian groups for any $p$.

Many families of groups for which \eqref{eq exp MG divides exp G} has been proven consist of regular $p$-groups, at first abelian $p$-groups and $p$-groups of class lower than $p$ (cf.~\cite[p.~98]{berkovich} and \S\ref{section exp G divides exp MG}).
On the other hand, if $P=G/\mho(G)$ belongs to some of such classes, then $G$ is regular and it also satisfies \eqref{eq exp MG divides exp G}.
We refer to the first Hall criterion claiming that if $|P/\mho(P)|<p^p$, then $G$ is regular and it is said \emph{absolutely regular}.
\begin{ppsi}\label{pps regular p-groups}
If $G$ is a regular $p$-group and $\exp\M(G/\mho(G))$ divides $p$, then $G$ satisfies \eqref{eq exp MG divides exp G}.
Moreover, \eqref{eq exp MG divides exp G} holds for groups of exponent $p$ iff it holds for regular $p$-groups.
In particular, absolutely regular $p$-groups enjoy this property, and in general regular 3-groups.
\end{ppsi}
We shall now prove the bound concerning the derived length, since we obtain the result in its stronger versions,
first involving any subnormal series, then involving abelian 2-groups and powerful $p$-groups for $p>2$.
\begin{proof}[\bf Proof of Theorem \ref{thm derived length}]
By iteration of Theorem \ref{thm main}, if a group $G$ admits a subnormal series
\begin{equation}\label{eq subnormal series}
G=G_0 > G_1 > \cdots > G_{r-1} > G_r=1\ ,\ G_i\trianglelefteq G_{i-1}\ ,\ Q_i=G_{i-1}/G_{i}\ ,
\end{equation}
then $$\exp\GU(G)\ \ |\ \ \prod_{j=1}^r\exp\GU(Q_i)\ .$$
By Lemma \ref{lem abelian 2-groups and powerful p-groups}, we have respectively:

\vspace{0.2cm}
\noindent I. Let $G$ be a $p$-group for $p>2$. Assume $G$ admits a subnormal series \eqref{eq subnormal series} where $Q_i$ are powerful $p$-groups. Then
$$\exp\M(G)\ \ |\ \ \prod_{j=2}^r\exp Q_j\cdot\exp\M(Q_1)\ .$$
II. Let $G$ be a $2$-group. Assume $G$ admits a subnormal series \eqref{eq subnormal series} where $Q_i$ are abelian.
Then
$$\exp\M(G)\ \ |\ \ 2^{|I|}\cdot\prod_{j=2}^r\exp Q_j\cdot\exp\M(Q_1)\ .$$
where $I\subseteq\{2,\ldots,r\}$ is such that $k\in I$ iff $Q_k$ has a subgroup isomorphic with $C_{e_k}\times C_{e_k}$ for $e_k=\exp Q_{k}$.

\vspace{0.2cm}
These bounds prove Theorem \ref{thm derived length} considering the derived series, so that the factor $Q_i$ are abelian.
In case $p>2$ apply I, as abelian $p$-groups are powerful.
In case $p=2$ apply II and substitute $|I|$ with $d-1$.
Notice that $\exp Q_k$ divides $\exp G$ for every $k$, as well as $\exp\M(Q_1)$ divides $\exp G$ since $Q_1=G/G'$ is abelian.
\end{proof}
We expose our alternative proof that the study of the exponent of the multiplier can be restricted in some extent to 2-generator groups.
\begin{ppsi}\label{pps 2 generator subgroups cover}
Let $\Sp(G)$ denote the set of 2-generator subgroups of $G$, then
$$\exp\GU(G)\ \ |\ \ \underset{S\in\Sp(G)}{\lcm} \exp\GU(S)\ .$$
\end{ppsi}
For a fixed positive integer $n$, let $\Sp(n)$ denote the set of isomorphism classes of 2-generator groups whose exponent divides $n$.
Substituting $\Sp(G)$ with $\Sp(\exp G)$ in the bound of Proposition \ref{pps 2 generator subgroups cover}, we obtain a bound depending only on the exponent of the group (cf.~\cite[Pr.~2.4]{moravec}).

We assume that $\exp G=p^k$, by the Zel'manov solution of the restricted Burnside problem \cite{zelmanov}, \cite{zelmanov2}, there exists a finite group $$\Zel{k}=\mbox{RBP}(2,p^k)$$ such that every element in $\Sp(p^k)$ is a homomorphic image of $\Zel{k}$.
Therefore, by Theorem \ref{thm main} we have that $\exp\GU(S)$ divides $\exp\GU(\Zel{k})$ for every $S$ in $\Sp(p^k)$, and we can also add some information to this result.
\begin{ppsi}\label{pps exponent of zelmanov}
If $G$ is a group of exponent $p^k$, then $$\exp\GU(G)\ \ |\ \ \exp\GU(\Zel{k})\ .$$
Moreover, $$\exp\GU(\Zel{k})=p^k\cdot\exp\M(\Zel{k})$$ and $$p^{k}\ \ |\ \ \exp\M(\Zel{k})\ .$$
\end{ppsi}
Given an account on the theory of central extensions (\S\ref{section central extensions}) and projective representations (\S\ref{section proj representations}), we discuss a generalization of the Schur's construction which proves that covering groups always exist (\S\ref{section schur construction}), then we introduce the unitary cocycles which define the unitary cover (\S\ref{section technique}), finally we prove the encountered results (\S\ref{section proof of main theorems}).

\section{Background}\label{section background}
Let $G$ be a group, and $A$ an abelian group.
A \emph{2-cocycle} is a map $\alpha:G\times G\to A$ satisfying $$\alpha(x,y)\cdot\alpha(xy,z)=\alpha(x,yz)\cdot\alpha(y,z)\ .$$
A \emph{2-coboundary} is a cocycle obtained from a map $\zeta:G\to A$ as
$$\delta\zeta(x,y)=\zeta(x)\cdot\zeta(y)\cdot\inv{\zeta(xy)}\ .$$
The sets of cocycles and coboundaries are denoted with $Z^2(G,A)$ and $B^2(G,A)$ respectively, they constitute abelian groups under pointwise multiplication.
The quotient $\Hn^2(G,A)=Z^2(G,A)/B^2(G,A)$ is the \emph{second cohomology group}.
In the particular case $A=\C^\times$, we obtain the Schur multiplier $\M(G)=\Hn^2(G,\C^\times)$, and we briefly denote $Z^2(G)=Z^2(G,\C^\times)$ and $B^2(G)=B^2(G,\C^\times)$.

These definitions play a fundamental role in the theory of central extensions, and in the theory of projective representations.
We will give an account hereby, recommending the reading of \cite[pp.181-185]{isaacs}.
Accordingly with this reference we adopt the right notation $x^y=\inv{y}xy$ and $[x,y]=\inv{x}\inv{y}xy$.

\subsection{Central extensions and Schur covers}\label{section central extensions}
A \emph{central extension} of a group $G$ is a group $\Gamma$ having a central subgroup $A\leq Z(\Gamma)$ such that $\Gamma/A$ is isomorphic with $G$.
It is usually written as $$\omega:1\to A\to\Gamma\stackrel{\pi}{\to} G\to 1$$ where $A=\ker\pi$, and $\omega$ will be now defined.
Let $\phi:G\to\Gamma$ be a section, that is $\pi(\phi(g))=g$ for every $g\in G$.
By definition, every $\gamma\in\Gamma$ can be uniquely written as $\gamma=a\cdot\phi(g)$ for some $a\in A$ and some $g\in G$.
Then, $\omega:G\times G\to A$ is associated with $\phi$ by the relation $$\phi(g)\cdot\phi(h)=\omega(g,h)\cdot\phi(gh)\ ,$$ and in turn $\omega\in Z^2(G,A)$.

Consider now a different section $\phi':G\to\Gamma$, clearly $\phi'(g)=\zeta(g)\cdot\phi(g)$ for some $\zeta:G\to A$.
From the analogue relation defining $\omega'$, it follows that $\omega'=\omega\cdot\delta\zeta$.
Multiplication by a coboundary correspond to a change of section.
We may also mention that the trivial cocycle $G\times G\to \{1\}\leq A$ corresponds to the trivial extension $\Gamma=G\times A$.

We briefly show how the Schur multiplier parametrizes the central extensions.
Denote by $\check A=\Hom(A,\C^\times)$ the group of the irreducible characters of $A$.
Then there exists $\eta:\check A\to\M(G)$ called the \emph{standard map}, defined as $$\eta:\check A\to\M(G)\ \ ,\ \ \lambda\mapsto\eta(\lambda)=[\lambda\circ\omega]\ \ ,\ \ \lambda\circ\omega(x,y)=\lambda(\omega(x,y))\ .$$
By the discussion above $\eta$ is well-defined, and it is easy to see that $\eta$ is a homomorphism such that $$\ker\eta=(A\cap\Gamma')^\perp\ \ ,\ (A\cap\Gamma')^\perp=\{\chi\in\check A\ |\ A\cap\Gamma'\leq\ker\pi\}\ .$$

The standard map also leads to the definition of Schur covers: a central extension is a \emph{Schur cover} of $G$ if the standard map is an isomorphism.
An equivalent definition is the following: a \emph{Schur cover} of $G$ is a central extension such that the kernel is isomorphic with the Schur multiplier and it is contained in the derived subgroup, $$1\to M\to\Gamma\to G\to 1\ \ ,\ \ M\simeq\M(G)\ \ ,\ \ M\leq Z(\Gamma_G)\cap\Gamma_G'\ .$$
If we make the weaker assumption that the standard map is onto, then $\Gamma$ has the \emph{projective lifting property}.
This is equivalent to the following property, $$1\to A\to\Gamma\to G\to 1\ \ ,\ \ A\cap\Gamma'\simeq\M(G)\ \ ,\ \ A\leq Z(\Gamma_G)\ .$$

If $\Gamma$ has the projective lifting property, then $\exp\M(G)$ has to divide $\exp\Gamma$.
Therefore, it has interest to find a minimal bound for the exponent of an extensions with the projective lifting property. We remark that this lower bound has not to be realized by a Schur cover, as shown by the following example.
\begin{es}
Consider the semidirect product of two cyclic groups of order $p^2$ defined by $$G=\gen{x,y\ |\ x^{p^2}=y^{p^2}=1,\ y^x=y^{p+1}}\ .$$
It can be seen, for instance using {\tt Gap} \cite{gap}, that $\exp\M(G)=p$ and that the group $$\Gamma_1=\gen{\bar x,\bar y\ |\ \bar x^{p^2}=\bar y^{p^3}=1,\ \bar y^{\bar x}=\bar y^{p+1}}$$ is the only Schur cover of $G$, and it has exponent $p^3$.
Nevertheless, the group $$\Gamma_2=\gen{\tilde x,\tilde y,\tilde z\ |\ \tilde x^{p^2}=\tilde y^{p^2}=\tilde z^{p^2}=1\ ,\ \tilde y^{\tilde x}=\tilde y^{p+1}\cdot\tilde z,\ [\tilde z,\tilde x]=[\tilde z,\tilde y]=1}$$ has the projective lifting property for $G$, and satisfies $\exp\Gamma_2=p^2$.
\end{es}
Among the central extensions those with the projective lifting property have fundamental importance, as they permit to transfer results on ordinary representations to projective representations and vice-versa. This was depicted by Schur who also proved by a constructive method that Schur covers always exist (\S\ref{section schur construction}).

Concerning the problem of bounding the exponent of the Schur multiplier, we can focus on the order of the single elements.
Therefore, we introduce a local variation: for any $\mu\in\M(G)$ we will say that a central extension $$1\to A_\mu\to\Gamma\to G\to 1$$ is a \emph{$\mu$-cover} if $A_\mu$ is cyclic and the standard map $\eta_\mu$ maps $\check A_\mu$ onto $\gen{\mu}$.
This definition is not usually stated, as for any $\mu\in\M(G)$ it is possible to obtain a $\mu$-cover as a quotient of any extension with the projective lifting property.
\begin{pps}\label{pps cover mu-extensions}
Let $\Gamma$ be an extension of $G$ with the projective lifting property, and $\mu\in\M(G)$. Then a $\mu$-cover $\Gamma_\mu$ can be obtained as a quotient of $\Gamma_G$.
In particular, $\exp\Gamma_\mu$ divides $\exp\Gamma_G$.
\end{pps}
\begin{proof}
Since the standard map $\eta_G:\check A\to\M(G)$ is assumed to be onto, there exists a preimage $\lambda\in\check A$ of $\mu$ under $\eta_G$, that is $\eta_G(\lambda)=\mu$.
We claim that the $\mu$-cover is $\Gamma_\mu=\Gamma_G/\ker\lambda$, whose exponent divides $\exp\Gamma_G$.
Set $A_\mu=A/\ker\lambda$, then $\lambda$ can be identified with a faithful irreducible character $\lambda_\mu$ of the cyclic group $A_\mu$, and the standard map $\eta_\mu:\check A_\mu\to \gen{\mu}$ is onto.
\end{proof}
We write down some complementary formulas for further reference.
Let $\Gamma$ be any central extension.
For any section $\phi$, an element $\gamma$ of $\Gamma$ is uniquely written as $\gamma=a\cdot\phi(g)$ for some $a\in A$ and some $g\in G$.
Hence, $o(\gamma)$ divides $\lcm\{o(a),o(\phi(g))\}$ and since $o(\phi(g))=o(g)\cdot o(\phi(g)^{o(g)})$ it holds
\begin{equation}\label{eq exponent of extensions}
 \exp\Gamma=\lcm\{\exp A\ ,\ \max_{g\in G} o(g)\cdot o(\phi(g)^{o(g)})\}\ .
\end{equation}
Moreover, for $g\in G$ it is not difficult to see that
\begin{equation}\label{eq order power central extension}
\phi(g)^{o(g)}=\prod_{j=0}^{o(g)-1}\omega(g,g^j)\ .
\end{equation}
Finally, concerning conjugation in $\Gamma$, by comparison of $\phi(x)\cdot\phi(y)$ and $\phi(y)\cdot\phi(x^y)$ it follows that
\begin{equation}\label{eq conjugation central extension}
\phi(x)^{\phi(y)}=\omega(x,y)\inv{\omega(y,x^y)}\cdot\phi(x^y)
\end{equation}
holds for any $x,y\in G$.
\begin{pps}\label{pps central extension and generators}
Let $\Gamma$ be a central extension of a $d$-generator group $G$.
Then there exists a $d$-generator sugroup $X$ of $\Gamma$, which is a central extension of $G$ such that $X'=\Gamma'$.
In particular, if $\Gamma$ has the projective lifting property, then also $X$ does, and if $\Gamma$ is a Schur cover, then $X=\Gamma$.
\end{pps}
\begin{proof}
Let $G=\gen{x_1,\ldots,x_d}$ and $\phi:G\to\Gamma$ be any section.
We claim that the desired subgroup is $X=\gen{\phi(x_1),\ldots,\phi(x_d)}$.
For $g\in G$ fix a writing $g=x_{i_1}^{\varepsilon_1}\cdots x_{i_l}^{\varepsilon_l}$, then $\phi(g)=b\cdot\phi(x_{i_1})^{\varepsilon_1}\cdots\phi(x_{i_l})^{\varepsilon_l}=b\cdot\xi$, for $b\in A$ and $\xi\in X$.
Any $\gamma\in\Gamma$ is uniquely written as $a\cdot\phi(g)$ for some $a\in A$ and $g\in G$.
Therefore, since $A\leq Z(\Gamma)$, then $\Gamma'=\gen{[\gamma_1,\gamma_2]\ |\ \gamma_i\in\Gamma}=\gen{[\xi_1,\xi_2]\ |\ \xi_i\in X}=X'$.
\end{proof}

\subsection{Projective representations and twisted group algebras}\label{section proj representations}
In analogy to the group algebra $\C[G]$ for ordinary representations, for projective representations it is defined the \emph{twisted group algebra}, which in turn relies on the cocycles.
For $\alpha\in Z^2(G)$, $\C^\alpha[G]$ is the $\C$-algebra with basis $\bar G=\{\bar g\ |\ g\in G\}$ identified with the group, and product $\bar x\cdot \bar y=\alpha(x,y)\cdot \Ccg{xy}$ obeying to the group product unless a twisting coefficient.

The cocycle condition is the associative law $(\bar x\cdot \bar y)\cdot\bar z=\bar x\cdot(\bar y\cdot \bar z)$, whereas multiplication by a coboundary represents a locally-linear change of group-basis $\tilde g=\zeta(g)\cdot \bar g$.
As common we consider normalized cocycles, that is $\alpha(1,1)=1$.
The meaning of this assumption is that $\bar 1$ is the identity of the twisted group algebra.
Hence, for normalized coboundaries $\delta\zeta$ it can be assumed $\zeta(1)=1$.

For a subgroup $H$ of $G$, the \emph{restriction map} is defined by $$\res:\M(G)\to\M(H)\ ,\ [\alpha]\mapsto[\alpha_H]\ ,\  \alpha_H(h_1,h_2)=\alpha(h_1,h_2)\ ,$$ and there is a natural identification $\C^{\alpha_H}[H]\leq\C^{\alpha}[G]$.
Then, for a normal subgroup $N$ of $G$ the \emph{inflation map} is defined by $$\inf:\M(G/N)\to\M(G)\ ,\ [\beta]\mapsto[\beta^\ast]\ ,\  \beta^\ast(g_1,g_2)=\beta(g_1 N,g_2 N)\ .$$

Clearly, the image of the inflation map from $M(G/N)$ is contained in the kernel of the restriction to $\M(N)$, a description of these subgroups can be done in terms of the idempotents of $\C^{\alpha_N}[N]$.
We recall that the twisted group algebra $\C^\alpha[G]$ is semi-simple: it admits a decomposition in irreducible subspaces, each one defined by an idempotent.

It can be seen that $[\alpha_H]=1$ iff $\C^{\alpha_H}[H]$ admits a 1-dimensional idempotent.
Moreover, for a normal subgroup $N$ of $G$, it was proven by R.~J.~Higgs \cite[Pr.~1.5]{higgs degree 1} that $[\alpha]=[\beta^\ast]$ for some $\beta\in Z^2(G/N)$ iff $\C^{\alpha_N}[N]$ admits a 1-dimensional idempotent which is invariant under conjugation in $\C^{\alpha}[G]$.
In analogy to \eqref{eq conjugation central extension}, for any $x,y\in G$ comparing $\bar x\cdot\bar y$ and $\bar y\cdot\Ccg{x^y}$ we have the relation
\begin{equation}\label{eq conjugation twisted group algebra}
 \bar x^{\bar y}=\alpha(x,y)\cdot\inv{\alpha(y,x^y)}\cdot\Ccg{x^y}
\end{equation}
which describes conjugation in $\C^\alpha[G]$.
\begin{pps}\label{pps conditions on cocycles}
Let $N\trianglelefteq G$ and $\alpha\in Z^2(G)$.
If $\alpha_N=1$ and $\alpha(n,g)=\alpha(g,n^g)$ for every $n\in N$ and $g\in G$, then $[\alpha]$ is inflated from $G/N$.
If in addition $G=N\rtimes H$ and $[\alpha_H]=1$, then it holds $[\alpha]=1$.
\end{pps}
\begin{proof}
Since $\alpha_N=1$, then $\C^{\alpha_N}[N]$ admits the principal idempotent $$\idempotent{N}=\frac{1}{|N|}\sum_{n\in N}\bar n\ ,$$ which is invariant in $\C^\alpha[G]$ by \eqref{eq conjugation twisted group algebra}.
In case $G=N\rtimes H$, since we assume $[\alpha_H]=1$, then $\C^{\alpha_H}[H]$ admits a central 1-dimensional idempotent $\upsilon_H$.
As in the general case, $\C^{\alpha_N}[N]$ admits the principal idempotent $\idempotent{N}$.
By \eqref{eq conjugation twisted group algebra} $\idempotent{N}$ and $\upsilon_H$ commutes, so that $\idempotent{N}\cdot\upsilon_H$ is a 1-dimensional idempotent of $\C^{\alpha}[G]$, and $[\alpha]=1$.
\end{proof}
Also for the powers there is a formula analogue to \eqref{eq order power central extension}, as for any $g\in G$ it holds
\begin{equation}\label{eq order twisted group algebra}
 \bar g^{o(g)}=\prod_{j=0}^{o(g)-1}\alpha(g,g^j)\ .
\end{equation}
Cocycles whose group-basis satisfy the identity $\bar g^{o(g)}=1$ for any $g\in G$ will play the main role in the next section.

\section{Method}
\subsection{Schur's construction}\label{section schur construction}
We abstract the fundamental tool for our main results.
We give a generalization of the construction which proves Schur's theorem on the existence of a covering group, this will lead to the definition of the unitary covers (\S\ref{section technique}).
\begin{dfn}
Let $H$ be a finite subgroup of $Z^2(G)$.
We define a central extension $$1\to\check H\to\omegatimes{H}{G}\to G\to 1\ \ ,\ \ \check H\leq Z(\omegatimes{H}{G})\ .$$
The underlying set of $\omegatimes{H}{G}$ is $G\times H$, and multiplication is given by the rule $$(g,\chi)\cdot(h,\psi)=(gh,\omega(g,h)\cdot\chi\psi)\ ,$$ where $\omega(g,h)\in\check H$ is defined by $\omega(g,h)(\alpha)=\alpha(g,h)$ for $\alpha\in H$.
\end{dfn}
The proof of Schur's theorem is done in this terms: since $B^2(G)$ is a divisible subgroup of finite index in $Z^2(G)$, then it has a complement $Z^2(G)=B^2(G)\oplus J$, and $\omegatimes{J}{G}$ is a Schur cover of $G$ (cf.~\cite[Th.~11.17]{isaacs}).

This construction is natural respect to the standard map in the following sense. For a cyclic decomposition $H=\gen{\alpha_1}\oplus\ldots\oplus\gen{\alpha_k}$, the dual group admits the decomposition $$\check H=\gen{\check\alpha_1}\oplus\ldots\oplus\gen{\check\alpha_k}\ \ ,\ \ \check\alpha_i(\alpha_j)=\left\{\begin{array}{cl}e^{2\pi\iota/o(\alpha_i)}&\mbox{if }j=i\\ 1&\mbox{otherwise}\end{array}\right.\ \ ,\ \ \iota=\sqrt{-1}\ \ ,$$ and there is a canonical identification of $H$ with the double dual $$H\check\ \check\ \equiv H\ \ ,\ \ \alpha_i\check\ \check\ \equiv\alpha_i\ \ ,$$ under this identification the standard map relative to $\omegatimes{H}{G}$ is the projection from $Z^2(G)$ to $\M(G)$ $$\eta:H\check\ \check\ \equiv H\to\M(G)\ ,\ \eta(\alpha)=[\alpha]\ .$$
Referring to \S\ref{section central extensions}, we immediately have the following lemma.
\begin{lem}\label{lem omega product and proj lifting ppt}
The extension $\omegatimes{H}{G}$ has the projective lifting property iff every cocycle in $Z^2(G)$ is cohomologous with a cocycle of $H$, and it is a Schur cover iff in addition $H\cap B^2(G)=1$.
\end{lem}
For a pair of subgroups $K\leq H\leq Z^2(G,A)$, there is a natural isomorphism $$\check K\to\check H/K^\perp\ \ ,\ \ K^\perp=\{\chi\in\check H\ |\ K\leq\ker\chi\}\ ,$$ defined choosing coherent cyclic decompositions
$$\left\{\begin{array}{ll}
      K=\gen{\beta_1}\oplus\ldots\oplus\gen{\beta_l}&\\
      H=\gen{\alpha_1}\oplus\ldots\oplus\gen{\alpha_l}\oplus\ldots\oplus\gen{\alpha_k}&,\ \ \beta_i\in\gen{\alpha_i}\ ,
  \end{array}\right.$$
then setting $\check\beta_i\mapsto\check\alpha_i K^\perp$, this induces an isomorphism $$\omegatimes{K}{G}\simeq(\omegatimes{H}{G})/K^\perp\ \ .$$
A case of particular interest is when $K$ is obtained via the inflation map.
\begin{lem}\label{lem duality and inflation}
Let $N$ be a normal subgroup of $G$, $H$ be a finite subgroup of $Z^2(G)$, and $L$ be a finite subgroup of $Z^2(G/N)$ such that $H\cap Z^2(G/N)^\ast=H\cap L^\ast$. Denote $$K=H\cap L^\ast\ \ ,\ \ \dot N=\gen{(n,1_H)\ |\ n\in N}\ ,$$ then there is a surjection $$\omegatimes{L}{(G/N)}\twoheadrightarrow(\omegatimes{H}{G})/K^\perp\dot N\ ,$$ which is an isomorphism in case $L^\ast\leq H$.
\end{lem}
\begin{proof}
If $(1,\chi)\in\dot N\cap\check H$, then $\chi=\prod_{j=1}^k\omega(n_{1,j},n_{2,j})$ for some $n_{i,j}\in N$, so that $\chi\in K^\perp$. Consequently, $$\dot N\cap\check H\leq K^\perp\ .$$
Also, $(n,1_H)^{(g,\chi)}=(n^g,\omega(n,g)\inv{\omega(g,n^g)})$ and $\omega(n,g)\inv{\omega(g,n^g)}\in K^\perp$.
Therefore, $$K^\perp\dot N\trianglelefteq \omegatimes{H}{G}\ .$$
Therefore, we obtain the central extension $$1\to\check H/K^\perp\to(\omegatimes{H}{G})/K^\perp\dot N\to G/N\to 1$$ and we show that $(\omegatimes{H}{G})/K^\perp\dot N$ is a homomorphic image of $\omegatimes{L}{(G/N)}$.
Write
$$\left\{\begin{array}{ll}
    L=\gen{\gamma_1}\oplus\ldots\oplus\gen{\gamma_l}\oplus\ldots\oplus\gen{\gamma_m}&\\
    K=\gen{\beta_1}\oplus\ldots\oplus\gen{\beta_l}&,\ \ \beta_i\in\gen{\gamma_i^\ast}\\
    H=\gen{\alpha_1}\oplus\ldots\oplus\gen{\alpha_l}\oplus\ldots\oplus\gen{\alpha_k}&,\ \ \beta_i\in\gen{\alpha_i}\ .
\end{array}\right.$$
If $\beta_i=(\gamma_i^\ast)^{m_i}$, then there is an isomorphism $$\check L/\tilde K^\perp\simeq\check K\ \ ,\ \ \check\gamma_i\tilde K^\perp\mapsto\check\beta_i\ \ ,\ \ \tilde K=\gen{\gamma_1^{m_1}}\oplus\ldots\oplus\gen{\gamma_l^{m_l}}\ ,$$ which can be composed with the canonical isomorphism $\check K\simeq\check H/K^\perp$ to give $$\check L\twoheadrightarrow\check L/\tilde K^\perp\simeq\check H/K^\perp\ \ ,\ \ \check\gamma_i\mapsto\check\alpha_i K^\perp\ .$$
Then, observe that $(gn,1_H)=(g,1_H)\cdot(n,\inv{\omega(g,n)})$ and that $(n,\inv{\omega(g,n)})\in K^\perp\dot N$.
The map $$(gN,\check\gamma_1^{k_1}\cdots\check\gamma_m^{k_m})\mapsto(g,\check\alpha_1^{k_1}\cdots\check\alpha_l^{k_l})\ K^\perp\dot N$$ is well defined, and it is the desired homomorphism.
In case $L^\ast\leq H$, then $K=L^\ast$ and $\check L\simeq\check K$. Thus, the map described is one to one.
\end{proof}

\subsection{Unitary cocycles and unitary covers}\label{section technique}
We introduce a subgroup of $Z^2(G)$, whose definition is done accordingly to \eqref{eq order twisted group algebra}, and we mimic the Schur's construction (\S\ref{section schur construction}) introducing the unitary cover.
\begin{dfn}\label{dfn unitary cocycles}
A cocycle $\alpha\in Z^2(G)$ is said to be \emph{unitary} if $$\prod_{j=0}^{o(g)-1}\alpha(g,g^j)=1$$ for every $g\in G$.
The set of unitary cocycles constitutes a group denoted by $\ZU(G)$, the \emph{unitary cover} of $G$ is the extension $$\GU(G)=\ZU(G)\check\ \propto G\ \ .$$
\end{dfn}
Unitary cocycles and unitary covers are the core of our main results. We begin proving that every cocycle is cohomologous with an unitary, so that cohomology can be done with unitary cocycles exclusively. We describe the unitary coboundaries and provide a relation wich refers to conjugation as shown by \eqref{eq conjugation twisted group algebra}.

At once we will show a clear benefit of these facts, as we ready give the first statement of Theorem \ref{thm main} in its explicit formulation.
\begin{lem}\label{lem unitary cocycles}
Let $\alpha\in Z^2(G)$. Then:
\begin{enumerate}[i)]
 \item There exists $\beta\in\ZU(G)$ such that $[\alpha]=[\beta]$. Therefore, $$M(G)\simeq\ZU(G)/\BU(G)$$ where $$\BU(G)=B^2(G)\cap\ZU(G)$$ is the group of unitary coboundaries.
 \item If $\delta\zeta\in\BU(G)$ for $\zeta:G\to\C^\times$, then $\zeta(g)^{o(g)}=1$ for any $g\in G$.
 \item $\ZU(G)$ and $\BU(G)$ are finite, and $\exp\BU(G)$ divides $\exp G$.
 \item If $\beta\in\ZU(G)$, then $\beta(x,g)^{o(x)}=\beta(g,x^g)^{o(x)}$ for every $x,g\in G$.
\end{enumerate}
\end{lem}
\begin{proof}
i) let $\alpha$ be any cocycle, define $\xi(g)$ to be any $o(g)$-root of $\prod_{j=0}^{o(g)-1}\alpha(g,g^j)$, set $\beta=\alpha\cdot\inv{\delta\xi}$, then $\beta$ is the unitary cocycle cohomologous with $\alpha$.
ii) apply the definition of unitary cocycle to $\delta\zeta$. iii) follows from ii.
iv) in $\C^\beta[G]$ it holds $(\bar x^{\bar g})^{o(x)}=(\beta(x,g)\inv{\beta(g,x^g)}\cdot\Ccg{x^g})^{o(x)}=[\beta(x,g)\inv{\beta(g,x^g)}]^{o(x)}\cdot(\Ccg{x^g})^{o(x)}$ by \eqref{eq conjugation twisted group algebra}.
Since $\beta\in\ZU(G)$, then $(\bar x^{\bar g})^{o(x)}=(\bar x^{o(x)})^{\bar g}=1$, and since $o(x^g)=o(x)$, then $(\Ccg{x^g})^{o(x)}=1$.
Therefore, $[\beta(n,g)\inv{\beta(g,n^g)}]^{o(x)}=1$ proving the assertion.
\end{proof}
\begin{lem}\label{lem unitary covers}
The exension $\GU(G)$ has the projective lifting property for $G$, and it satisfies $$\exp\GU(G)=\lcm\{\exp\ZU(G),\exp G\}\ .$$
Moreover, if $\Gamma$ is a central extension of $G$ having the projective lifting property, then $\exp\GU(G)$ divides $\exp\Gamma$.
\end{lem}
\begin{proof}
That $\GU(G)$ has the projective lifting property it follows by Lemma \ref{lem unitary cocycles} and Lemma \ref{lem omega product and proj lifting ppt}.
To find the exponent we use \eqref{eq exponent of extensions} with the section $\phi(g)=(g,1_A)$, where $A=\ZU(G)$, by definition $$\phi(g)^{o(g)}=(g^{o(g)},\prod_{j=1}^{o(g)-1}\omega(g,g^j))=1$$ and the assertion is proven.
We now prove minimality dividing the proof in two steps: first we prove a local-version, then we use this to complete the proof.

\emph{Local-version.} Let $\mu\in\M(G)$, and $\Gamma_\mu$ be a $\mu$-cover.
If $\beta\in\ZU(G)$ is such that $\mu=[\beta]$, then $o(\beta)$ divides $\exp\Gamma_\mu$.

\emph{Step I}. It is enough to prove that there exists one cocycle $\beta$ with the required assertion, since two such cocycles differ by an unitary coboundary, and by Lemma \ref{lem unitary cocycles} unitary coboundaries have order dividing $\exp G$ thus dividing $\exp\Gamma_\mu$.
Since $\Gamma_\mu$ is a $\mu$-cover, the standard map $\eta_\mu:\check A_\mu\to\gen{\mu}$ is onto.
Therefore, there exists $\lambda\in\check A_\mu$ such that $\eta_\mu(\lambda)=\mu$, where $\eta_\mu(\lambda)=[\lambda\circ\omega_\mu]$.
Reading the proof of Lemma \ref{lem unitary cocycles} together with \eqref{eq order power central extension}, an unitary cocycle $\beta$ cohomologous to $\lambda\circ\omega_\mu$ is found defining $\zeta:G\to\C^\times$ to be for $g\in G$ any $o(g)$-root of $\lambda(\phi(g)^{o(g)})$, then setting $\beta=(\lambda\circ\omega_\mu)\cdot\inv{\delta\zeta}$.
We show the assertion by use of \eqref{eq exponent of extensions}.
Since $o(\omega_\mu)$ divides $\exp A_\mu$ and $\lambda$ is a homomorphism, then the order of $\lambda\circ\omega_\mu$ divides $\exp A_\mu$, and since $\lambda$ is faithful, then $o(\lambda(\phi(g)^{o(g)}))=o(\phi(g)^{o(g)})$.
Therefore, $$o(\zeta(g))=o(g)\cdot o(\lambda(\phi(g)^{o(g)}))=o(g)\cdot o(\phi(g)^{o(g)})\ ,$$ and $o(\delta\zeta)$ divides $\max_{g\in G} o(g)\cdot o(\phi(g)^{o(g)})$.
Thus, $o(\beta)$ divides $\lcm\{o(\lambda\circ\omega_\mu),o(\delta\zeta)\}$ which divides $\lcm\{\exp A_\mu,\max_{g\in G}o(g)\cdot o(\phi(g)^{o(g)})\}$ that by \eqref{eq exponent of extensions} is $\exp\Gamma_\mu$.

\emph{Step II.}
For $\mu\in\M(G)$, by Lemma \ref{lem unitary cocycles} there exists $\beta_\mu\in\ZU(G)$ such that $[\beta_\mu]=\mu$.
By Proposition \ref{pps cover mu-extensions} there exists a $\mu$-cover $\Gamma_\mu$ obtained as a quotient from $\Gamma_G$, then by the local-version $o(\beta_\mu)$ divides $\exp\Gamma_\mu$ and consequently $\exp\Gamma_G$.
In particular, defining $J=\gen{\beta_\mu\ |\ \mu\in\M(G)}$, then $J$ is a subgroup of $\ZU(G)$ of exponent dividing $\exp\Gamma_G$.
Clearly $\ZU(G)=J\BU(G)$, so that $\exp\ZU(G)=\lcm\{\exp J,\ \exp\BU(G)\}$ which divides $\lcm\{\exp\Gamma_G,\ \exp\BU(G)\}$.
The proof is complete  by Lemma \ref{lem unitary cocycles} since $\exp\BU(G)$ divides $\exp\Gamma_G$.
\end{proof}

\subsection{Proof of the main theorems}\label{section proof of main theorems}
\begin{proof}[\bf Proof of Theorem \ref{thm main}]
The first statement is part of Lemma \ref{lem unitary covers}.
The proof of i) and iv) is explicitly written while proving ii) and v), nevertheless, we shall give an independent proof only based on commonly known results.

{\bf i)} We prove that $$ \exp\M(G)\ \ |\ \ \lcm\{\exp\ZU(N),\exp N\}\cdot\exp\M(G/N)\ ,$$ then we apply Lemma \ref{lem unitary covers}.
For any co-class $[\alpha]\in\M(G)$, we can assume by Lemma \ref{lem unitary cocycles} that $\alpha\in\ZU(G)$.
For $r=\lcm\{\exp\ZU(N),\exp N\}$, we show that $\alpha^r$ satisfies the first two conditions of Proposition \ref{pps conditions on cocycles}, proving that $[\alpha]^r$ is inflated from $\M(G/N)$ so that it becomes trivial when risen to the $\exp\M(G/N)$-power.
Since $\exp\ZU(N)$ divides $r$ clearly $(\alpha_N)^r=1$, and since $\exp N$ divides $r$ by Lemma \ref{lem unitary cocycles} the proof is complete.

{\bf iv)} We assume $G=N\rtimes H$, and we prove that $$\exp\M(G)\ \ |\ \ \lcm\{\exp\ZU(N),\ \exp N,\ \exp\M(H)\}\ ,$$ then we apply Lemma \ref{lem unitary covers}.
By a result of K.~Tahara \cite[Th.~2]{tahara}, $\M(G)$ is isomorphic with the direct sum of $\M(H)$ and the kernel of the restriction from $\M(G)$ to $M(H)$.
Denote by $r$ be the lcm, since $\exp\M(H)$ divides $r$ we can consider only co-classes $[\alpha]\in\M(G)$ whose restriction to $H$ is trivial.
By Lemma \ref{lem unitary cocycles} we can also assume that $\alpha\in\ZU(G)$.
The third condition of Proposition \ref{pps conditions on cocycles} is immediate, and since $\lcm\{\exp\ZU(N),\exp N\}$ divides $r$ the first two conditions follow the general case.

{\bf ii)}  We prove that $$\exp\ZU(G)\ \ |\ \ \lcm\{\exp\ZU(N),\exp N\}\cdot\lcm\{\exp\ZU(G/N),\exp G/N\}\ ,$$ then we apply Lemma \ref{lem unitary covers}.
Fix $\alpha\in\ZU(G)$ and a transversal $\mathcal{T}$ for $N$ in $G$. Define $\xi(g)=\alpha(t,n)$ for $g=tn$ where $t\in\mathcal T$ and $n\in N$.
Let $r=\lcm\{\exp\ZU(N),\exp N\}$, define $\beta=\alpha^r$ and $\tilde\beta=\beta\cdot\delta\vartheta$ for $\vartheta=\xi^r$.
We show that $\vartheta(g)^{o(gN)}=1$ for every $g\in G$, and that $\tilde\beta$ is inflated from $\ZU(G/N)$, as $o(\beta)$ divides $\lcm\{o(\tilde\beta),o(\delta\vartheta)\}$ this completes the proof.
For $g_1,g_2\in G$, let $g_i=t_in_i$ where $t_i\in\mathcal{T}$ and $n_i\in N$, and let $t_1t_2=t_{1,2}n_{1,2}$ where $t_{1,2}\in\mathcal{T}$ and $n_{1,2}\in N$.
In the twisted group algebra $\C^\beta[G]$ consider $$\vartheta(g)^{o(gN)}\cdot \bar g^{o(gN)}=(\vartheta(g)\cdot \bar g)^{o(gN)}=(\bar t\cdot\bar n)^{o(gN)}=\bar t^{o(gN)}\cdot\bar n^{\bar t^{o(gN)-1}}\cdots\bar n^{\bar t}\cdot\bar n\ .$$
For any $x\in G$ since $\bar x^{o(x)}=(\Ccg{x^{o(xN)}})^{o(x^{o(xN)})}=1$, then $$\prod_{j=1}^{o(xN)-1}\alpha(x,x^j)^{o(x^{o(xN)})}=1\ \ ,\ \ x\in G\ \ ,$$ in particular $\bar g^{o(gN)}=\Ccg{g^{o(gN)}}$ and $\bar t^{o(gN)}=\Ccg{t^{o(gN)}}$ as $tN=gN$.
Since $\exp N$ divides $r$,  by Lemma \ref{lem unitary cocycles} then $\bar n^{\bar t^j}=\Ccg{n^{t^j}}$ for every $j$, and since $\exp\ZU(N)$ divides $r$, then $\beta_N=1$.
Therefore, $$\bar t^{o(gN)}\cdot\bar n^{\bar t^{o(gN)-1}}\cdots\bar n^{\bar t}\cdot\bar n=\Ccg{t^{o(gN)}\cdot n^{t^{o(gN)-1}}\cdots n^t\cdot n}=\Ccg{g^{o(gN)}}$$ proving that $\vartheta(g)^{o(gN)}=1$.
We now prove that $\tilde\beta=\gamma^\ast$ for some $\gamma\in\ZU(G/N)$.
Notice that $\xi(t)=\xi(n)=1$ for any $t\in\mathcal{T}$ and any $n\in N$, so that $$\begin{array}{rl}\alpha(g_1,g_2)\cdot\delta\xi(g_1,g_2)&=\alpha(t_1,t_2)\cdot\delta\xi(t_1,t_2)\ \cdot\\ &\ \cdot\ [\alpha(n_1,t_2)\cdot\inv{\alpha(t_2,n_1^{t_2})}]\cdot\alpha(n_{1,2},n_1^{t_2}n_2)\cdot\alpha(n_1^{t_2},n_2)\end{array}\ .$$
Hence, define $\gamma\in Z^2(G/N)$ by setting $\gamma(g_1N,g_2N)=\tilde\beta(g_1,g_2)$, since $$\prod_{j=1}^{o(gN)-1}\tilde\beta(g,g^j)=\prod_{j=1}^{o(gN)-1}\alpha(g,g^j)^r\cdot\delta\vartheta(g,g^j)=\prod_{j=1}^{o(gN)-1}\delta\vartheta(g,g^j)=\vartheta(g)^{o(gN)}=1\ ,$$ then it holds $\gamma\in\ZU(G/N)$.

{\bf v)}  Assume $G=N\rtimes H$, and choose the transversal $\mathcal{T}=H$ so that $$\alpha(g_1,g_2)\cdot\delta\xi(g_1,g_2)=\alpha(h_1,h_2)\cdot[\alpha(n_1,h_2)\cdot\inv{\alpha(h_2,n_1^{h_2})}]\cdot\alpha(n_1^{h_2},n_2)\ .$$
The proof that $o(\xi(g))$ divides $\exp\GU(N)\cdot o(gN)$ follows the general case, and this gives the bound $$\exp\ZU(G)\ \ |\ \ \lcm\{\exp\ZU(H),\exp\GU(N)\cdot\exp H\}$$ where $\GU$ can replace $\ZU$  by Lemma \ref{lem unitary covers} with no loss.

{\bf iii)} Since $\ZU(G/N)^\ast\leq\ZU(G)$, then $$\GU(G/N)\simeq\GU(G)/(\ZU(G/N)^\ast)^\perp\dot N$$ by Lemma \ref{lem duality and inflation}.
\end{proof}
\begin{proof}[\bf Proof of Lemma \ref{lem abelian 2-groups and powerful p-groups}]
We begin finding a cover of minimal exponent for an abelian $p$-group $A$. Write the cyclic decomposition \eqref{eq cyclic decomposition}, then $$\Gamma=\gen{x_1,\ldots,x_m\ |\ o(x_i)=p^{d_i}\ ,\ [[x_i,x_j],x_k]=1}$$ is a cover for $A$ (cf.~\cite[p.~325]{karpilovsky}), which satisfies $\exp\Gamma=\exp A$ for $p>2$, and $\exp\Gamma=2^\sigma\cdot\exp A$ for $p=2$.
It can be seen that $\exp\ZU(A)=\exp\Gamma$, and by Lemma \ref{lem unitary covers} it follows that $\exp\GU(A)=\exp\Gamma$.
The case $p=2$ is proved, while for $p>2$ we use inductively this result on abelian $p$-groups.
By definition, $G'\leq\mho(G)$ so that $G/\mho(G)$ is elementary abelian, and we can assume that $\exp G>p$.
Since $\exp\GU(G)=\exp\GU(\mho(G))\cdot\exp\GU(G/\mho(G))$, by Theorem \ref{thm main}, the result follows by induction as $\mho(G)$ is powerful and $\exp\mho(G)=\exp G/p$ \cite[Cor.~1.5, Pr.~1.7]{lubotzkymann}.
\end{proof}
\begin{proof}[\bf Proof of Proposition \ref{pps regular p-groups}]
We can assume that $\exp G\geq p>2$. It is known that $\mho(G)$ is powerful \cite[p.~497]{lubotzkymann}, by Theorem \ref{thm main} and Lemma \ref{lem abelian 2-groups and powerful p-groups}, then $\exp\M(G)$ divides $\exp\mho(G)\cdot\exp\M(G/\mho(G))$.
Moreover, every element of $\mho(G)$ is a $p$-power \cite[Th.~7.2]{berkovich}, so that $\exp\mho(G)=\exp G/p$.
We remind that $\exp G/\mho(G)=p$, and groups of exponent $p$ are regular.
Assuming that $\exp\M(G/\mho(G))$ divides $p$ the proof is completed.
This is the case for $p=3$, as well for absolutely regular $p$-groups as $\cl(G/\mho(G))<p$.
\end{proof}
\begin{proof}[\bf Proof of Proposition \ref{pps 2 generator subgroups cover}]
Since there exists a $R\in\Sp(G)$ such that $\exp R=\exp G$, by Lemma \ref{lem unitary covers} it is enough to prove that there exists $S\in\Sp(G)$ such that $\exp\ZU(G)$ divides $\exp\ZU(S)$.
Choose any element $\mu\in\M(G)$ satisfying $o(\mu)=\exp\M(G)$. By Lemma \ref{lem unitary cocycles}, there exists $\alpha\in\ZU(G)$ such that $\mu=[\alpha]$.
Let $x,y\in G$ such that $o(\alpha)=o(\alpha(x,y))$, and set $S=\gen{x,y}$.
Since $\alpha_S\in\ZU(S)$ and $o(\alpha_S)=o(\alpha)$, then $o(\mu)$ divides $o(\alpha)$ and the proof is complete.
\end{proof}
\begin{proof}[\bf Proof of Lemma \ref{pps exponent of zelmanov}]
For any $p$-group $G$ and any integer $m$, the $m$-agemo subgroup is defined as $$\mho^m(G)=\gen{g^{p^m}\ \ |\ \ g\in G}\ .$$
Let $\Gamma$ be any central extension of $\Zel{k}$, by Proposition \ref{pps central extension and generators} there exists a 2-generated subgroup $X$ of $\Gamma$ which is a central extension of $\Zel{k}$ such that $\Gamma'=X'$.
Since $\Zel{k}$ is the maximal 2-generated group of exponent $p^k$, it follows that $X/\mho^k(X)\simeq\Zel{k}$. Therefore, $A\cap\Gamma'=A\cap X'\leq\mho^k(X)$.
The assertion $$\exp\GU(\Zel{k})=p^k\cdot\exp\M(\Zel{k})$$ follows, since $\GU(\Zel{k})$ has the projective lifting property.

We now prove that $p^k$ divides $\exp\M(\Zel{k})$.
For any integer $l$, the sequence $$1\to\mho^k(\Zel{k+l})\to\Zel{k+l}\to\Zel{k}\to 1$$ give rise to the central extension $$1\to\mho^k(\Zel{k+l})/[\mho^k(\Zel{k+l}),\Zel{k+l}]\to\Zel{k+l}/[\mho^k(\Zel{k+l}),\Zel{k+l}]\to\Zel{k}\to 1\ .$$
There is an embedding (\S\ref{section central extensions}) $$\mho^k(\Zel{k+l})\cap\Zel{k+l}'/[\mho^k(\Zel{k+l}),\Zel{k+l}]\hookrightarrow\M(\Zel{k})\ ,$$ which for large enough $l$ becomes an isomorphism, and the exponent of this group is $p^{m_0}$ for the minimal $m_0$ such that $$\mho^{m_0}(\mho^k(\Zel{k+l})\cap\Zel{k+l}')\leq[\mho^k(\Zel{k+l}),\Zel{k+l}]\ .$$
Any 2-generator group $G$ of exponent $p^{k+l}$ is a homomorphic image of $\Zel{k+l}$, then $$\mho^{m_0}(\mho^k(G)\cap G')\leq[\mho^k(G),G]\ ,$$ therefore, lower bounds for $\exp\M(\Zel{k})$ derives from two generators groups.
Consider the covering group of $C_n\times C_n$ defined by $$G=\gen{x,y\ |\ x^{p^{k}}=y^{p^{2k}}=1\ ,\ y^x=y^{p^k+1}}\ ,$$ then $G'=\mho^k(G)=Z(G)=\gen{y^k}$, so that $$\mho^{m_0}(\mho^k(G)\cap G')=\mho^{m_0}(\gen{y^k})\ \ ,\ \ [\mho^k(G),G]=1\ ,$$ so that $m_0\geq k$ completing the proof.
\end{proof}

\section*{Acknowledgement}
The author is conducing his PhD studies, he is indebted to his menthors Prof.~Eli Aljadeff (Technion) and Dr.~Yuval Ginosar (University of Haifa). 
A preliminary version was presented at "Groups St Andrews 2013", University of St Andrews, Scotland.

\bibliographystyle{alpha}

\end{document}